\pdfoutput=1

\documentclass[a4paper]{article}

\usepackage{amsmath}
\usepackage{amsthm}
\usepackage{amsfonts}
\usepackage{graphicx}

\theoremstyle{definition}
\newtheorem{prop}{Proposition}
\newtheorem{conj}{Conjecture}

\theoremstyle{remark}

\newcommand{\cM}{\mathcal{M}}

\newcommand{\cO}{\mathcal{O}}
\newcommand{\cU}{\mathcal{U}}
\newcommand{\cK}{\mathcal{K}}
\newcommand{\cS}{\mathcal{S}}
\newcommand{\sn}{\sqrt{n}}

\author{Jean-Paul Cardinal}
\title{Some experiments on the growth \\of Mertens matrices.}

\begin{document}

%~ generates the title
\maketitle
%~ insert the table of contents
%~ \tableofcontents

\begin{abstract}
We give some experimental observations on the growth of the norm of certain matrices related to the Mertens function. The results obtained in these experiments convince us that linear algebra may help in the study of Mertens function and other arithmetic functions.
\end{abstract}

%~ \listoffigures

\section{Introduction}
 Two families $\cU_n$ and $\cM_n$ of matrices related to the Mertens function have been introduced in \cite{jp_linalg_french}, \cite{jp_linalg}. The Riemann hypothesis would hold if the $l_2$-norm of the matrices $\cM_n$ was not growing too fast, which is confirmed by numerical evidence. Lagarias and Montague \cite{lagarias_montague} show that one can use the Frobenius norm as well. Following these authors we call Mertens matrices the matrices $\cM_n$ and refer the reader to the papers \cite{jp_linalg_french}, \cite{jp_linalg}, \cite{lagarias_montague} for a precise definition.
 
 This paper gives some numerical results about the growth of $\| \cM_n \|/\sn$. In Section \ref{norm_bounds} we plot the largest eigenvalues of $\cM_n$ and discuss their behavior according to the shape of the corresponding eigenvectors. In section \ref{preconditioning} we introduce a new family of matrices obtained by preconditioning $\cU_n$ and state a conjecture about the growth of their norm. This conjecture implies the Riemann hypothesis. Then we show that these matrices can be seen as finite-rank approximations of a compact operator on $L^2([0, 1])$.

\section{Empirical growth of $\| \cM_n \|$}
\label{norm_bounds}
If we denote by $M()$ the Mertens function, then it is shown in \cite[Theorem 24]{jp_linalg} that the inequality $\vert M(n) \vert \le \| \cM_n \|$ holds for all $n \in \mathbb{N}^\ast$. The Riemann hypothesis would then hold if 
$$\forall \epsilon > 0, \| \cM_n \| =O(n^{1/2+\epsilon})$$
As the Mertens function can grow only by $\sn$ on an interval of size $\sn$ it is sufficient to have the inequality above satisfied for $n$ being a square. So the integer $\sn$ will be subjected to the condition $n = k^2$ throughout this paper.

The growth of $\| \cM_n \|$ is not very clear when $\| \cM_n \|/\sn$ is displayed against $n$, as shown in \cite[Fig.2]{jp_linalg}. A better view is obtained when $\| \cM_n \|$ is displayed against $\sn$ as in Figure \ref{fig:figure_1} below ($n$ is restricted to the form $n = k^2$ in this experiment).
\begin{figure}[!ht]
\flushleft
\includegraphics[width=1.1\textwidth, height=0.85\textwidth]{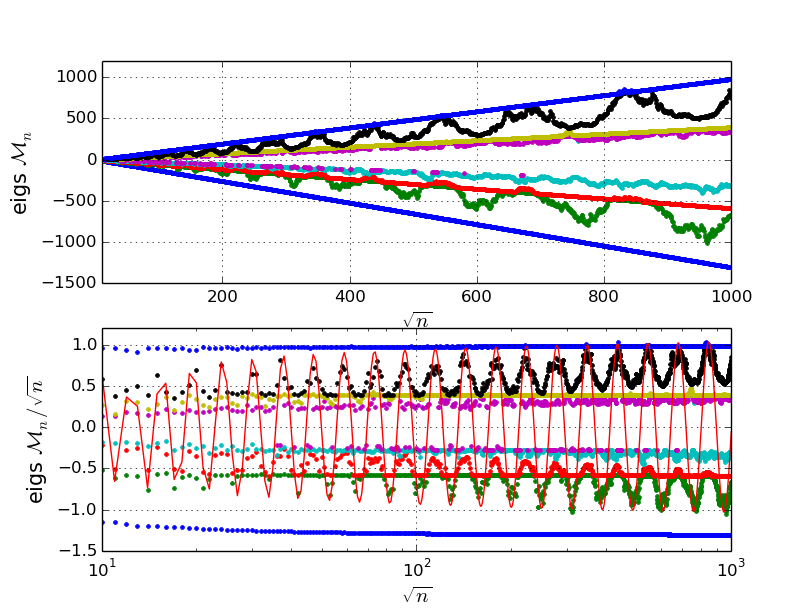}
\caption{Eight largest eigenvalues of $\cM_n$}
\label{fig:figure_1}
\end{figure}

In the first subplot we can see that the dominant eigenvalue, the bottom blue curve, grows steadily and close to a straight line. Some eigenvalues, but not all, follow the same regularity. The eigenvalues that grow irregularly show an oscillatory pattern. In the second subplot we use a semilog scale on the abscisses and add the curve (in red) $f(n) = 1.05\cos(14.14\log n -2.2)\sqrt{\log \log \log n} $. The angular frequency $14.14$ of the cosinus has been choosed as the imaginary part of the first zero of the Riemann Zeta function. As explained by Kotnik and van de Lune \cite{KL}, the quantity  $0.36\cos(14.14\log n - 1.69)$ is the first term of a series giving $M(n)/\sn$, by a theorem established in Titchmarsh \cite{T}. The extra factor $\sqrt{\log \log \log n} $ has been proposed by these authors as an asymptotic tendency of the quantity $M(n)/\sn$ as $n$ tends to infinity. As seen on Figure \ref{fig:figure_1}, our choice of $f(n)$ fits well the oscillations of the eigenvalues. Due to the presence of the factor $\sqrt{\log \log \log n}$ we are convinced that the largest oscillating eigenvalue (either the green or the black curve) will ultimately dominate the others, a fact that might be related to the results of Section \ref{bound_iK} below.

Figure \ref{fig:figure_2} shows the eigenvectors associated to the eight largest eigenvalues of $\cM_n$. On these plots we can see that quiet eigenvalues correspond to smooth eigenvectors and oscillating eigenvalues correspond to non-smooth eigenvectors. It looks like there are two kinds of eigenvalues, grouped by pairs of opposite sign.
\begin{figure}[!ht]
\flushleft
\includegraphics[width=1.1\textwidth, height=0.85\textwidth]{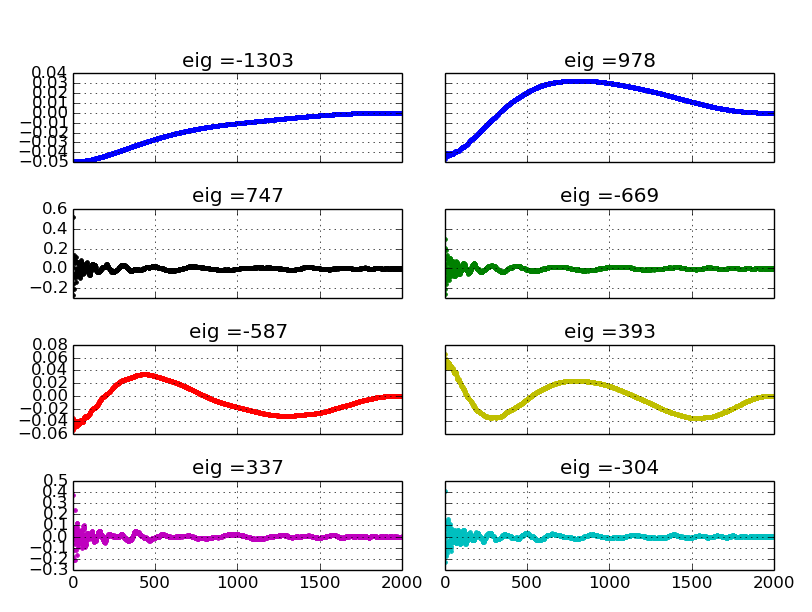}
\caption{Eight dominant eigenvectors of $\cM_n$}
\label{fig:figure_2}
\end{figure}

\section{A new family of matrices}
\label{preconditioning}
We now construct a family of matrices $\cK_n$ obtained by preconditioning $\cU_n$.
\subsection{Motivation}
For each integer $n$ let $\cS$ be the set of distinct integers of the form $\{ \lfloor n/k \rfloor : 1\le k \le n\}$. Lagarias and Montague \cite[Def. 2.1]{lagarias_montague} introduce the notation $\cS = \cS^- \cup \cS^+$ with $ \cS^- = \{j : 1\le j \le \lfloor \sn \rfloor\}$ and $\cS^+ = \{\lfloor n/j \rfloor : 1\le j \le \lfloor \sn \rfloor\}$. Depending on the value $n$ the subsets $\cS^-, \cS^+$ are disjoint or have only the element $\lfloor \sn \rfloor$ in common. The size of $\cS$ is the same as the size of the matrix $\cU$. From this we define \cite[Lemme 3.5]{jp_linalg_french} the vector $d = (\frac{\sn}{k})_{k\in\cS}$ and the diagonal matrix $D=\mathrm{diag}(d)$. Let $u$ be the vector of ones, of size $\#\cS$. We know from \cite[Proposition 2.24]{jp_linalg} that $M(n) = u^T \cU^{-1} u$. Then we introduce the preconditioned matrix $\cK = D^{-1/2}\cU D^{-1/2}$, thus
\begin{equation}
\label{wKw}
 M(n) = u^T D^{-1/2} D^{1/2} \cU^{-1} D^{1/2} D^{-1/2} u = w^T \cK^{-1} w
 \end{equation}
with $w = d^{-1/2}$ (the power $d^{-1/2}$ must be understood componentwise). Now we compute an estimate of $\|w\|^2  = \|w^-\|^2 + \|w^+\|^2$, $w^-$ and $w^+$ being the parts of $w$ indexed respectively by the subsets $\cS^-$ and $\cS^+$ that we can assume to be disjoint since on the contrary the final estimate would be the same,
 $$\begin{array}{lllllll}
\|w^-\|^2 & = & \frac{1}{\sn}\sum_{1\le j \le \lfloor \sn \rfloor} j & \sim & \frac{\sn}{2} &\\
\|w^+\|^2 & = & \frac{1}{\sn}\sum_{1\le j \le \lfloor \sn \rfloor} \lfloor \dfrac{n}{j} \rfloor & \le & \frac{1}{\sn}\sum_{1\le j \le \lfloor \sn \rfloor}  \dfrac{n}{j}  & \sim \frac{\sn}{2}\log n
\end{array}$$
From these estimates and (\ref{wKw}) we get the bound
\begin{equation}
\label{MK}
| M(n) | = \cO (\|\cK^{-1}\|)\sn\log n 
\end{equation}

\pagebreak
\subsection{Empirical norm bounds on the matrix $\cK_n^{-1}$}
\label{bound_iK}
The bound (\ref{MK}) incites to investigate the growth of $\|\cK^{-1}\|$.
Figure~\ref{fig:figure_3} shows the eight largest eigenvalues of $\cK_n^{-1}$, displayed against $\sn$.
\begin{figure}[!ht]
\flushleft
\includegraphics[width=1.1\textwidth, height=0.85\textwidth]{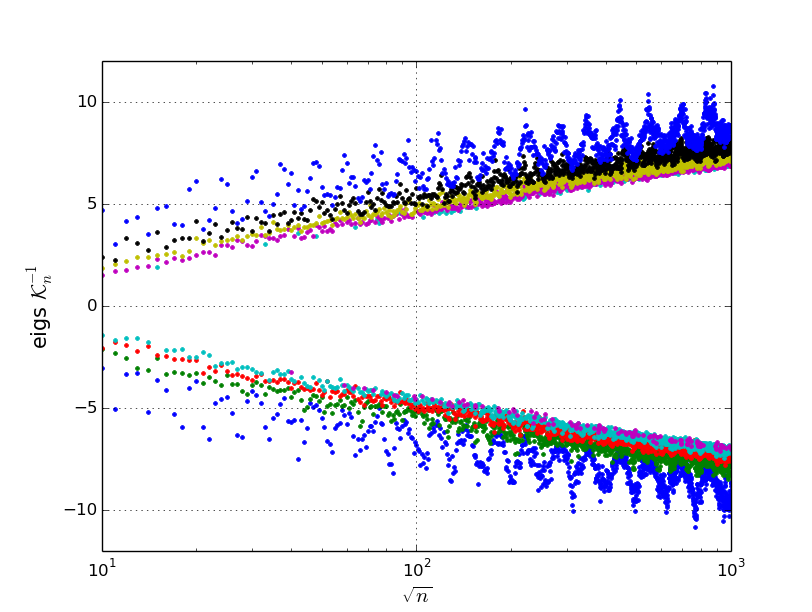}
\caption{Eight largest eigenvalues of $\cK_n^{-1}$}
\label{fig:figure_3}
\end{figure}
In this plot we can see that the eigenvalues oscillate, as in Figure \ref{fig:figure_1}, and that there are no more quiet eigenvalues. This confirms the fact that the oscillating eigenvalues eventually play the main role in the growth of $ \|\cK_n^{-1}\|$. The smoothness of the dominant eigenvalues, observed in Figure \ref{fig:figure_1}, might be an effect of the two-side multiplication of $\cU_n^{-1}$ by $T$ in the equality $\cM_n = T\cU_n^{-1}T$, $T$ being the matrix of the same size than $\cU_n$ and $\cM_n$, whose entries are ones above the antidiagonal and zeros below (for details we refer to \cite{jp_linalg}).

More important is the fact that the norm $\|\cK_n^{-1}\|$ seems to increase not much faster than a linear function of $\log n$. This empirical evidence, combined to the estimate (\ref{MK}), suggests the following conjecture, implying the Riemann hypothesis 
\begin{conj}
$$\forall \epsilon > 0, \| \cK_n^{-1} \| =O(n^{\epsilon})$$
\end{conj}

\subsection{$\cK_n$ as a finite-rank approximate of a compact operator}
In this section we show that $\cK_n$ can be seen as a finite-rank approximation of a compact operator on $L^2([0, 1]$. For that purpose we give a sligthly different construction of the matrix $\cU_n$. Let the function $f$ be defined by $f(t) = 1/(2t)$ if $t \le 0.5$ and $f(t) = 2(1-t)$ if $t > 0.5$. This function is decreasing and maps $]0, 1]$ onto $[0, +\infty[$. An easy calculation shows that $\cU = (\cU_{ij})_{1 \le i, j \le 2\sn-1}$ with $$\cU_{ij} = \lfloor f(\frac{i}{2\sn})f(\frac{j}{2\sn}) \rfloor$$ 
Notice that the indexes do not run in $\cS$ but in the range of contiguous integers $1, \cdots, 2\sn-1$. In other words the matrix $\cU_n$ can be seen as a finite-rank approximate of the operator $U$ defined on $L_2([0, 1])$ by the kernel $u(s, t) = \lfloor f(s)f(t) \rfloor$, $(s, t) \in [0, 1]^2$. Likewise, we have $\cK = (\cK_{ij})_{1 \le i, j \le 2\sn-1}$  with $$\cK_{ij} = f(\frac{i}{2\sn})^{-1/2}\lfloor f(\frac{i}{2\sn})f(\frac{j}{2\sn}) \rfloor f(\frac{j}{2\sn})^{-1/2}$$
Thus, as above, $\cK_n$ can be seen as a finite-rank approximate of the operator $K$ defined on $L_2([0, 1])$ by the kernel $k(s, t) = f(s)^{-1/2}\lfloor f(s)f(t) \rfloor f(t)^{-1/2}$.
\begin{prop}
$K$ is a compact operator.
\end{prop}
\begin{proof}
Let $\epsilon > 0$ and consider the linear operator $K_\epsilon$ defined on $L^2([0, 1]$ by the kernel $k_\epsilon(s, t) = f(s)^{-1/2-\epsilon}\lfloor f(s)f(t) \rfloor f(t)^{-1/2-\epsilon}$. We have
 $$\begin{array}{lllll}
k_\epsilon(s, t)^2 & \le & f(s)^{-1-2\epsilon} (f(s)f(t)) ^{2} f(t)^{-1-2\epsilon}  
 & = & f(s)^{1-2\epsilon} f(t)^{1-2\epsilon}
\end{array}$$
Integrating this inequality on the unit square gives
 $$\begin{array}{lllllll}
\| k_\epsilon \|_{L^2} & = & \left( \int_0^1\int_0^1 k_\epsilon(s, t)^2 \mathrm{d}s\mathrm{d}t \right)^{1/2} & \le & \int_0^1 f(s)^{1-2\epsilon}\mathrm{d}s  
 & \sim & \dfrac{1}{4\epsilon}
\end{array}$$
and thus $K_\epsilon$ is a Hilbert-Schmidt operator. Therefore, since we have $\| k_\epsilon - k \|_{L^2} \to 0$ as $\epsilon \to 0$, $K$ is compact.
\end{proof}

\end{document}